\documentclass[12pt]{amsart}
\usepackage[T2A]{fontenc}
\usepackage[cp1251]{inputenc}
\usepackage[english]{babel}
\usepackage{amsmath,latexsym,amsthm,amsfonts,tipa,upgreek}
\usepackage{amssymb,amscd,graphpap, stmaryrd}

\newcommand\NN{{\mathbb N}}

\newcommand\RR{{\mathbb R}}
\newcommand\w{{\omega}}
\newcommand\ta{{\overline{\tau}}}

\newcommand\Tau{\mathcal{T}}

\newcommand\PP{{\mathcal P}}
\newcommand\FF{{\mathcal F}}

\newtheorem{Th}{Theorem}[section]

\newtheorem{Qs}{Question}[section]
\newtheorem{Cr}{Corollary}[section]
\theoremstyle{definition}

\begin{document}

\title{Relative size of subsets of a semigroup}
\author{Igor Protasov and Serhii Slobodianiuk}
\subjclass[2010]{05D10, 20M12, 22A15.}
\keywords{large, thick and prethick subsets of a semigroup; ultrafilters; the Stone-$\check{C}$ech compactification of a discrete semigroups; minimal left ideal.}
\date{}
\address{Department of Cybernetics, Kyiv University, Volodymyrska 64, 01033, Kyiv, Ukraine}
\email{i.v.protasov@gmail.com; }
\address{Department of Mechanics and Mathematics, Kyiv University, Volodymyrska 64, 01033, Kyiv, Ukraine}
\email{slobodianiuk@yandex.ru}
\maketitle

\begin{abstract} Given a semigroup $S$, we introduce relative (with respect to a filter $\tau$ on $S$) versions of large, thick and prethick subsets of $S$, give the ultrafilter characterizations of these subsets and explain how large could be some cell in a finite partition of a subset $A\in\tau$. \end{abstract}

\section{Introduction}
For a semigroup $S$, $a\in S$, $A\subseteq S$ and $B\subseteq S$, we use the standard notations
$$a^{-1}B=\{x\in S:ax\in B\}, \ \ A^{-1}B=\bigcup_{a\in A}a^{-1}B,$$
$$[A]^{<\w}=\{F\subseteq A, F\text{ is finite}\}.$$
A subset $A$ of $S$ is called
\begin{itemize}
\item{} {\em large} if there exists $F\in [S]^{<\w}$ such that $S=F^{-1}A$;
\item{} {\em thick} if, for every $F\in [S]^{<\w}$, there exists $x\in S$ such that $Fx\subseteq A$;
\item{} {\em prethick} if $F^{-1}A$ is thick for some $F\in[S]^{<\w}$;
\item{} {\em small} if $L\setminus A$ is large for any large subset $L$.
\end{itemize}
In the dynamical terminology \cite[p.101]{b8}, large and prethick subsets are known as syndedic and piecewise syndedic.
These and several other combinatorially rich subsets of a semigroup are intensively studied in connection with the Ramsey Theory (see \cite[Part III]{b8}). In \cite{b6}, large, thick and prethick subsets are called right syndedic, right thick and right piecewise syndedic.

The names large and small subsets of a group appeared in \cite{b4}, \cite{b5} with additional adjective "left". Unexplicitely, thick subsets were used in \cite{b11} to partition an infinite totally bounded topological group $G$ into $|G|$ dense subsets.
For more delicate classification of subsets of a group by their size, we address the reader to \cite{b3}, \cite{b9}, \cite{b10}, \cite{b14}, \cite{b17}, \cite{b18}.

In frames of general asymptology \cite[Chapter 9]{b20}, large and thick subsets of a group could be considered as counterparts of dense and open subsets of a topological space.

Our initial motivation to this note was a desire to refine and generalize to semigroup the following statement \cite[Corollary 3.4]{b13}: if a neighborhood $U$ of the identity $e$ of a topological group $G$ is finitely partitioned then there exists a cell $A$ of the partition and a finite subset $F\subset U$ such that $FAA^{-1}$ is a neighborhood of $e$. On this way, we run to some relative (with respect to a filter) versions of above definitions.

Let $S$ be a semigroup and let $\tau$ be a filter on $S$. We say that a subset $A$ of $S$ is
\begin{itemize}
\item{} {\em $\tau$-large} if, for every $U\in\tau$, there exists $F\subseteq[U]^{<\w}$ such that $F^{-1}A\in\tau$;
\item{} {\em $\tau$-thick} if there exists $U\in\tau$ such that, for any $F\in[U]^{<w}$ and $V\in\tau$, one can find $x\in V$ such that $Fx\subseteq A$;
\item{} {\em $\tau$-prethick} if, for every $U\in\tau$, there exists $F\in[U]^{<\w}$ such that $F^{-1}A$ is $\tau$-prethick;
\item{} {\em $\tau$-small} if $L\setminus A$ is $\tau$-large for every $\tau$-large subset $L$.
\end{itemize}
In the case $\tau=\{S\}$, we omit $\tau$ and get the seminal classification of subsets of $S$ by their size.

To conclude the introduction, we need some algebra in the Stone-$\check{C}$ech compactifications from \cite{b8}.

For a discrete semigroup $S$. we consider the Stone-$\check{C}$ech compactification $\beta S$ of $S$ as the set of all ultrafilters on $S$, identifying $S$ with the set of all principal ultrafilters, and denote $S^*=\beta S\setminus S$.
For a subset $A$ of $S$ and a filter $\tau$ on $S$, we set
$$\overline{A}=\{p\in\beta S:A\in p\},\ \ \ta=\bigcap\{\overline{A}:A\in\tau\}=\{p\in\beta S:\tau\subseteq p\}$$
and note that the family $\{\overline{A}:A\subseteq S\}$ forms base for the open sets on $\beta S$, and each non-empty closed subset in $\beta S$ is of the form $\overline{\tau}$ for an appropriate filter $\tau$ on $S$.

The universal property of the Stone-$\check{C}$ech compactifications of discrete spaces allows to extend multiplication from $S$ to $\beta S$ in such  way that, for any $p\in\beta S$ and $g\in S$ the shifts $x\mapsto xp$ and $x\mapsto gx$, $x\in \beta S$ are continuous.

For any $A\subseteq S$ and $q\in \beta S$, we denote
$$A_q=\{x\in S:x^{-1}A\in q\}.$$
Then formally the product $pq$ of ultrafilters $p$ and $q$ can be defined \cite[p.89]{b8} by the rule: $$A\in pq\leftrightarrow A_q\in p.$$
In this note, we give the ultrafilter characterizations of $\tau$-large and $\tau$-thick subsets (section~$2$) and $\tau$-prethick subsets (section~$3$) in spirit of \cite{b6}, \cite{b8}, \cite{b18}. If $\tau$ is a subsemigroup of $\beta S$, we describe the minimal left ideal of $\overline{\tau}$ to understand how big could be the cells in a finite partition of a subset $A\in\tau$.

\section{Relatively large and thick subsets}~\label{s2}
Let $\tau$ be a filter on a semigroup $S$.
\begin{Th}\label{t2.1}
A subset $L$ of $S$ is $\tau$-large if and only if, for every $p\in\overline{\tau}$ and $U\in\tau$, we have $L_p\cap U\neq\varnothing$.
\end{Th}
\begin{proof}
We suppose that $L$ is $\tau$-large and take arbitrary $p\in\ta$ and $U\in\tau$.
We choose $F\in[U]^{<\w}$ such that $F^{-1}L\in\tau$.
Since $F^{-1}L=\bigcup_{g\in F} g^{-1}L$, there exists $g\in F$ such that $g^{-1}F\in p$ so $g\in L_p$ and $L_p\cap U\neq\varnothing$.

To prove the converse statement, we assume that $L$ is not $\tau$-large and choose $U\in\tau$ such that $F^{-1}L\notin\tau$ for every $F\in[U]^{<w}$.
Then we take an ultrafilter $p\in\ta$ such that $G\setminus F^{-1}L\in p$ for each $F\in[U]^{<\w}$. Clearly, $g^{-1}L\notin p$ for every $g\in U$ so $U\cap L_p=\varnothing$.
\end{proof}
\begin{Th}\label{t2.2}
A subset $T$ of $S$ is $\tau$-thick if and only if there exists $p\in\ta$ such that $T_p\in\tau$.
\end{Th}
\begin{proof}
We suppose that $T$ is $\tau$-thick and pick corresponding $U\in\tau$. The set $[U]^{<w}\times\tau$ is directed $\le$ by the rule:
$$(F,V)\le(F',V')\leftrightarrow F\subseteq F', V'\subseteq V.$$
For each pair $(F,V)$ we choose $g(F,V)\in V$ such that $Fg(F,V)\subseteq T$.
The family
$$P_{F,V}=\{g(F',V'):(F,V)\le (F',V')\}, \ \ (F,V)\in [U]^{<\w}\times \tau$$
is contained in some ultrafilter $p\in\ta$.
By the construction, $U\subseteq T_p$ so $T_p\in\tau$.

To prove the converse statement, we choose $p\in\ta$ such that $T_p\in\tau$.
Given any $F\in[T_p]^{<\w}$ and $V\in\tau$, we take $P\in p$ such that $P\subseteq V$ and $gP\subseteq T$ for each $g\in F$. Then we choose an arbitrary $x\in P$ and get $Fx\subseteq T$ so $T$ is $\tau$-thick.
\end{proof}
We say that a subset $T$ of $S$ is {\em $\tau$-extrathick} if $T_p\in\tau$ for each $p\in\ta$. 

By \cite[Theorem 2.4]{b6}, a subset $T$ is thick if and only if $T$ intersects each large subset non-trivially.
In the case $\tau=\{G\}$, this is a partial case of the following theorem.
\begin{Th}\label{t2.3}
If each subset $U\in\tau$ is $\tau$-extrathick then a subset $T$ of $S$ is $\tau$-thick if and only if $T\cap L\cap U\neq\varnothing$ for any $\tau$-large subset $L$ and $U\in \tau$.
\end{Th}
\begin{proof}
We assume that $T$ is $\tau$-thick and use Theorem~\ref{t2.2} to find $p\in\ta$ such that $T_p\in\tau$.
We take an arbitrary $\tau$-large subset $L$ and $U\in\tau$. Since $U$ is $\tau$-extrathick, we have $U_p\in\tau$. 
By Theorem~\ref{t2.1}, $L_p\cap(T_p\cap U_p)\neq\varnothing$.
If $g\in L_p\cap T_p\cap U_p$ then $L\in gp$, $T\in gp$, $U\in gp$. 
Hence, $T\cap L\cap U\neq\varnothing$.

We suppose that $T\cap L\cap U=\varnothing$ for some $\tau$-large subset $L$ and $U\in\tau$ but $T$ is $\tau$-thick. We take $p\in\ta$ such that $T_p\in\tau$. Since $U$ is $\tau$-extrathick, we have $U_p\in\tau$. By Theorem~\ref{t2.1}, $L_p\cap(T_p\cap U_p)\neq\varnothing$.
If $g\in L_p\cap T_p\cap U_p$ then $L\in gp$, $T\in gp$, $U\in gp$. 
Hence, $T\cap L\cap U\neq\varnothing$ and we get a contradiction.
\end{proof}
\begin{Th}\label{t2.4}
Let $g\in S$ and let $\tau$ be a filter on $S$ such that $g^{-1}U\in\tau$ for each $U\in\tau$.
If a subset $L$ of $S$ is $\tau$-large and a subset $T$ of $S$ is $\tau$-thick then $gL$ and $g^{-1}T$ are $\tau$-large and $\tau$-thick respectively.
\end{Th}
\begin{proof}
To prove that $gL$ is $\tau$-large, we take an arbitrary $U\in\tau$ and choose $V\in\tau$ such that $gV\subseteq U$ (using $g^{-1}U\in\tau$).
Since $L$ is $\tau$-large, there is $F\in[V]^{<\w}$ such that $F^{-1}L\in\tau$.
We note that $F^{-1}L=(gF)^{-1}gF$.
Since $gF\in[U]^{<\w}$, we conclude that $gF$ is $\tau$-large.

To see that $g^{-1}T$ is $\tau$-thick, we pick $U\in\tau$ such that, for every $F\in[U]^{<\w}$ and $W\in\tau$, there is $x\in W$ such that $Fx\subseteq T$.
We choose $V\in\tau$ such that $gV\subseteq U$.
Then we take an arbitrary $H\in[V]^{<w}$ and $W\in\tau$.
Since $gH\in[U]^{<\w}$, there exists $y\in W$ such that $gHy\subseteq T$ so $Hy\subseteq g^{-1}T$ and $g^{-1}T$ is $\tau$-thick.
\end{proof}
We say that a family $\FF$ of subsets of $S$ is {\em left (left inverse) invariant} if, for any $A\in\FF$ and $g\in S$, we have $gA\in\FF$ ($g^{-1}A\in\FF$).
\begin{Cr} If $\tau$ is inverse invariant then the family of all $\tau$-large ($\tau$-thick) subsets is left (left inverse) invariant.
\end{Cr}
\begin{Th}\label{t2.6}
Let $\tau$ be a filter on $S$ such that, for every $U\in\tau$, we have $\{g\in S:g^{-1}U\in\tau\}\in\tau$.
If $T$ is $\tau$-thick then there exists $V\in\tau$ such that $g^{-1}T$ is $\tau$-thick for every $g\in V$. 
\end{Th}
\begin{proof}
We take $U\in\tau$ such that, for any $K\in[U]^{<\w}$ and $W\in\tau$, we have $Kx\subseteq T$ for some $x\in W$.
Then we choose $V\in\tau$ such that, for every $g\in V$, there exists $V_g\in\tau$ with $gV_g\subseteq U$.
Given any $F\in[V_g]^{<\w}$ and $W\in\tau$, we pick $x\in W$ such that $gFx\subset T$ so $Fx\subseteq g^{-1}T$ and $g^{-1}T$ is $\tau$-thick.
\end{proof}
A topology $\Tau$ on a semigroup $S$ is called {\em left invariant} if each left shift $x\mapsto gx$, $g\in G$ is continuous (equivalently, the family $\Tau$ is left inverse invariant).

We assume that $S$ has identity $e$ and say that a filter $\tau$ on $S$ is {\em left topological} if $\tau$ is the filter of neighborhoods of $e$ for some (unique in the case if $S$ is a group) left invariant topology $\Tau$ on $S$.

Let $\tau$ be a left topological filter on $S$. 
Then each subset $U\in\tau$ is $\tau$-extrathick and $\tau$ satisfies Theorem~\ref{t2.6}.
Hence, Theorems~\ref{t2.3} and~\ref{t2.6} hold for $\tau$.

We show that Theorem~\ref{t2.6} needs not to be true with $\tau$-large subsets in place of $\tau$-thick subsets even if $\tau$ is a filter on neighborhoods of the identity for some topological group.

We endow $\RR$ with the natural topology, denote $\RR^+=\{r\in\RR, r>0\}$ and take the filter $\tau$ of neighborhoods of $0$.
The set $\RR^+$ is $\tau$-large because $\RR^+-x\in\tau$ for each $x\in\RR^+$.
On the other hand, $\RR^++x$ is not $\tau$-large for each $x\in\RR^+$.

\section{Relatively prethick subsets}
We say that a filter $\tau$ on $S$ is a {\em semigroup filter} if $\ta$ is a subsemigroup of the semigroup $\beta S$ and note that, if either $\tau$ is inverse left invariant or $S$ has the identity and $\tau$ is left topological then $\tau $ is a semigroup filter.

In the case $\tau=\{S\}$, the following statement is Theorem~$4.39$ from \cite{b8}.
\begin{Th}\label{t3.1}
Let $\tau$ be a semigroup filter on $S$.
An ultrafilter $p\in\ta$ belongs to some minimal left ideal $L$ of $\ta$ if and only if, for each $A\in p$, the set $A_p$ is $\tau$-large.
\end{Th}
\begin{proof}
Let $L$ be a minimal left ideal of $\ta$, $p\in L$, $A\in p$ and $U\in\tau$.
Clearly, $L=\ta p$.
We take an arbitrary $r\in\tau$.
By the minimality of $L$, $\ta rp=\ta p$, so there exists $q_r\in\tau$ such that $q_r rp=p$.
Since $A\in q_r rp$ and $U\in q_r$, by the definition of the multiplication in $\beta S$, there exists $B_r\in r$ such that $\overline{B}_rp\subseteq\overline{x_r^{-1}A}$.
We consider the open cover $\{\overline{B_r},r\in\ta\}$ of the compact space $\ta$ and choose its finite subcover 
$\{\overline{B}_r:r\in K\}$.
We put $B=\bigcup_{r\in K}B_r$, $F=\{x_r,r\in K\}$.
Then $B\in\tau$ and $B\subseteq (F^{-1}A)_p$. By the choice, $F\subseteq U$.
Since $p$ is an ultrafilter, we have $(F^{-1}A)_p=F^{-1}A_p$.
Hence, $A_p$ is $\tau$-large.

To prove the converse statement we suppose that $\ta p$ is not minimal and choose $r\in\ta$ such that $p\notin\ta rp$.
Since the subset $\tau rp$ is closed in $\ta$, there exists $A\in p$ with $\overline{A}\cap\ta rp =\varnothing$.
It follows that $A\notin qrp$ for every $q\in\ta$.
Hence, $S\setminus A\in qrp$ for every $q\in\ta$.
It follows that there exists $U\in\tau$ such that $x^{-1}(G\setminus A)\in rp$ for each $x\in U$.
By the assumption, there exists $F\in[U]^{<\w}$ such that $F^{-1}A\in qp$ for every $q\in\ta$.
In particular, $x^{-1}A\in rp$ for some $x\in F$ and we get a contradiction.
\end{proof}
\begin{Cr}\label{C1}
Let $\tau$ be a semigroup filter on $S$ and let $p\in\ta$ belongs to some minimal left ideal of $\ta$. Then every subset $A\in p$ is $\tau$-prethick.
\end{Cr}
\begin{proof}
Given an arbitrary $U\in\tau$, we use Theorem~\ref{t3.1} to find $F\in[V]^{<\w}$ such that $(F^{-1}A)_p\in\tau$. By Theorem~\ref{t2.2}, $F^{-1}A$ is $\tau$-thick. Hence, $A$ is $\tau$-prethick.
\end{proof}
\begin{Cr}\label{C2}
Let $\tau$ be a semigroup filter on a group $G$ and let $U\in\tau$.
Then, for every finite partition $\PP$ of $U$ and every $V\in\tau$, there exists $A\in\PP$ and $F\in[V]^{<\w}$ such that $F^{-1}AA^{-1}\in\tau$.
\end{Cr}
\begin{proof}
We take $p$ from some minimal left ideal of $\ta$.
Then we choose $A\in\PP$ such that $A\in p$. 
Applying Theorem~\ref{t3.1}, we find $F\in[V]^{<\w}$ such that $(F^{-1}A)_p\in\tau$.
If $x\in(F^{-1}A)_p$ then $F^{-1}A\in xp$ and $x\in F^{-1}AA^{-1}$.
Hence, $F^{-1}AA^{-1}\in \tau$.
\end{proof}

In connection with Corollary~\ref{C2}, we would like to mention one of the most intriguing open problem in the subset combinatorics of groups \cite[Problem 13.44]{b12} possed by the first author in $1995$: given any group $G$, $n\in\NN$ and partition $\PP$ on $G$ into $n$ cells, do there exit $A\in\PP$ and $F\subseteq G$ such that $G=FAA^{-1}$ and $|F|\le n$? For recent state of this problem see the survey \cite{b2}.

On the other hand \cite{b1}, if an infinite group $G$ is either amenable or countable, then for every $n\in\NN$, there exists a partition $G=A\cup B$ such that $FA$ and $FB$ are not thick for each $F$ with $|F|\le n$. We do not know whether such a $2$-partition exists for 	any uncountable group $G$ and $n\in\NN$.
\begin{Th}\label{t3.2} Let $G$ be a group, $\tau$ be a filter of neighborhoods of the identity for some group topology on $G$ and $U\in\tau$. Then, for any partition $\PP$ of $U$, $|\PP|=n$ and $V\in\tau$, there exist $A\in\PP$ and $K\subseteq V$ such that $KAA^{-1}\in\tau$ and $K\le2^{2^{n-1}-1}$.
\end{Th}
\begin{proof}
We consider only the case $n=2$. 
For $n>2$, the reader can adopt the inductive arguments from \cite[pp. 120-121]{b16}, where this fact was proved for $\tau=\{G\}$. 
So let $U=A\cup B$ and $e\in B$. We choose $W\in\tau$ such that $WW\subseteq U$ and denote $C=A\cap W$. 
If there exists $H\in\tau$ such that $xC\cap C\neq\varnothing$ for each $x\in H$ then $CC^{-1}\in\tau$ and we put $F=\{e\}$, so $F^{-1}AA^{-1}\in\tau$. 
Otherwise, we take $g\in V\cap W$ such that $gC\cap C=\varnothing$. 
Then $gC\subseteq WW\subseteq U$, so $gC\subseteq B$ and $B\cup g^{-1}B\in\tau$. 
We put $F=\{e,g\}$. Since $e\in B$, we have $F^{-1}BB^{-1}\in\tau$.
\end{proof}
Recall that a family $\FF$ of subsets of a set $X$ is {\em partition regular} if, for every $A\in\FF$ and any finite partition of $A$, at least one cell of the partition is a member of $\FF$.

For a subsemigroup filter $\ta$ on $S$, we denote by $M(\ta)$ the union of all minimal left ideals of $\ta$.
In the case $\tau=\{G\}$, the following statement is Theorem~$4.40$ from \cite{b8}.
\begin{Th}\label{t3.5}
Let $\tau$ be left inverse invariant filter on a semigroup $S$. 
Then the following statements hold
\begin{itemize}
\item[(i)] a subset $A$ of $S$ is $\tau$-prethick if and only if $\overline{A}\cap M(\ta)\neq\varnothing$;
\item[(ii)] $P\in cl M(\ta)$ if and only if each $A\in p$ is $\tau$-prethick;
\item[(iii)] the family of all $\tau$-prethick subsets of $S$ is partition regular.
\end{itemize}
\end{Th}
\begin{proof}
$(i)$ If $\overline{A}\cap M(\ta)\neq\varnothing$ then $A$ is $\tau$-prethick by Corollary~\ref{C1}. 

Assume that $A$ is $\tau$-prethick and pick a finite subset $F$ such that $F^{-1}A$ is $\tau$-thick.
We use Theorem~\ref{t2.2} to find $p\in\ta$ such that $(F^{-1}A)_p\in\tau$.
Then $F^{-1}A\in qp$ for every $q\in\ta$.
The set $\ta p$ contains some minimal left ideal $L$ of $\ta$.
We take any $r\in L$ so $F^{-1}A\in r$ and $A\in tr$ for some $t\in F$.
Since $\tau$ is inverse left invariant $tr\in\ta$. 
Hence, $tr\in M(\ta)\cap\overline{A}$.

The statements $(ii)$ and $(iii)$ follow directly from $(i)$.
\end{proof}
\begin{Th}\label{t3.6}
Let $\tau$ be a left invariant filter on a group $G$. 
A subset $A$ of $G$ is $\tau$-prethick if and only if $A$ is not $\tau$-small.
\end{Th}
\begin{proof}
By the definition and Theorem~\ref{t2.4}, the family of all $\tau$-small subsets of $G$ is left invariant and invariant under finite unions. We suppose that $A$ is $\tau$-small and $\tau$-prethick and take $K\in[G]^{<w}$ such that $KA$ is $\tau$-thick.
We note that $G$ is $\tau$-large and $KA$ is $\tau$-small so $G\setminus KA$ is $\tau$-large.
But $(G\setminus KA)\cap KA=\varnothing$ and we get a contradiction with Theorem~\ref{t2.3}.
\end{proof}
We do not know whether Theorems~\ref{t3.5} and~\ref{t3.6} hold for any left topological filter $\tau$ (even for filters of neighborhoods of identity of topological groups).

For a subset $A$ of an infinite group $G$, we denote
$$\Delta(A)=\{x\in G: xA\cap A\text{ is infinite}\}.$$
Answering a question from \cite{b15}, Erde proved \cite{b7} that if $A$ is prethick then $\Delta(A)$ is large.
We conclude the paper with some relative version of this statement.

For a filter $\tau$ on a semigroup $S$ and $A\subseteq S$, we denote
$$\Delta_\tau(A)=\{x\in S:(x^{-1}A\cap A)\cap U\neq\varnothing\text{ for any }U\in\tau\}.$$
In the case of a group $G$, $\Delta(A)=(\Delta(A))^{-1}$ so we have $\Delta(A)=\Delta_\tau(A)$ for the filter $\tau$ of all cofinite subsets of $G$.
\begin{Th}\label{t3.7}
Let $\tau$ be a left inverse invariant filter on a semigroup $S$.
If a subset $A$ of $S$ is $\tau$-prethick then $\Delta_\tau(A)$ is $\tau$-large.
\end{Th}
\begin{proof}
We observe that $\Delta_\tau(A)=\bigcup\{A_p:p\in\ta, A\in p\}$.
Now let $A$ be $\tau$-prethick. We use Theorem~\ref{t3.5}$(i)$ to find $p\in\overline{A}\cap M(\ta)$.
By Theorem~\ref{t3.1}, for every $U\in\tau$, there exists a finite subset $K\subseteq U$ such that $K^{-1}A_p\in\tau$. Since $A_p\subseteq\Delta_\tau(A)$, we have $K^{-1}\Delta_\tau(A)\in\tau$, so $\Delta_\tau(A)$ is $\tau$-large.
\end{proof}
Let $\tau$ be a left invariant filter on a group $G$ and let $X\subseteq G$.
Then $\Delta_\tau(X)=\{g\in G:(gX\cap X)\cap U\neq\varnothing\text{ for each } U\in\tau\}$ and $\Delta_\tau(X\setminus U)=\Delta_\tau(X)$ for each $U\in\tau$.
Now let $\tau$ be left invariant and $G\setminus K\in\tau$ for each $K\in[G]^{<\w}$. By \cite[Proposition 2.7]{b2}, for every $n$-partition $\PP$ of $G$, there exists $A\in\PP$ and $F\in[G]^{<\w}$ such that $F\Delta_\tau(A)\in\tau$ and $|F|\le n!$ 
This statement and above observations imply that, for any $U\in\tau$ and $n$-partition $\PP$ of $U$, there exist $F\in[G]^{<\w}$ and $A\in\PP$ such that $|F|\le n!$ and $F\Delta_\tau(A)\in\tau$.
Moreover, for any pregiven $V\in\tau$, $F$ can be chosen from $V^{-1}$.
Indeed, we take $x\in\bigcap_{g\in F}gV$ so $F^{-1}x\subseteq V$ and $x^{-1}F\Delta_\tau(A)\in\tau$.
\begin{Qs}\label{Q1}
Let $\tau$ be a filter of neighborhoods of the identity for some group topology on a group $G$ and let $U\in\tau$. Given any $n$-partition $\PP$ of $U$ and $V\in\tau$, do there exist $A\in\PP$ and $F\subseteq V$ such that $FAA^{-1}\in\tau$ and $|F|\le n!$
\end{Qs}
By Theorem~\ref{t3.2}, the answer to Question~\ref{Q1} is positive with $2^{2^n}$ in place of $n!$, .
\begin{Qs}\label{Q2}
Does there exist a function $f:\NN\to\NN$ such that for any group $G$, a filter $\tau$ of a group topology on $G$, $U\in\tau$ and an $n$-partition $\PP$ of $U$, there are $A\in\PP$ and $K\in[G]^{<\w}$ such that $K\Delta_\tau(A)\in\tau$ and $|K|\le f(n)$? If yes, then can $K$ be chosen from pregiven $V\in\tau$?
\end{Qs}
We conjecture the positive answer to Question~\ref{Q2} with $f(n)=2^{2^n}$ (or even with $f(n)=n!$).


\begin{thebibliography}{99}
\bibitem{b1} T.~Banakh, I.~Protasov, S.~Slobodianiuk, {\it Syndedic submeasures and partitions of $G$-spaces and groups}, Intern. J. Algebra Comp. {\bf 23} (2013), 1611--1623.
\bibitem{b2} T.~Banakh, I.~Protasov, S.~Slobodianiuk, {\it Densities, submeasures and partitions of $G$-spaces and groups}, Algebra Discrete Math. {\bf 17} (2014), N2, 193--221, preprint (http://arxiv.org./abs/1303.4612).
\bibitem{b3} T.~Banakh, I.~Protasov, S.~Slobodianiuk, {\it Scattered subsets of groups}, Ukr. Math. J. {\bf 65} (2015), 304--312, preprint (http://arxiv.org/abs/1312.6946).
\bibitem{b4} A.~Bella, V.~Malykhin, {\it On certain subsets of a group}, Questions, Answers Gen Topology {\bf 17} (1999), no.2, 183--197.
\bibitem{b5} A.~Bella, V.~Malykhin, {\it On certain subsets of a group II}, Questions, Answers Gen Topology {\bf 19} (2001), no.1, 81--94.
\bibitem{b6} V.~Bergelson, N.~Hindman, R.~McCutcheon, {\it Notes of size and combinatorial properties of quotient sets in semigroups}, Topology Proceedings {\bf 23} (1998), 23--60.
\bibitem{b7} J.~Erde, {\it A note on combinatorial derivation}, preprint (http:/arxiv.org/abs/1210.7622).
\bibitem{b8} N.~Hindman, D.~Strauss, Algebra in the Stone-$\check{C}$ech compactification, 2nd edition, de Gruyter, 2012.
\bibitem{b9} Ie.~Lutsenko, I.~Protasov, {\it Sparse, thin and other subsets of groups}, Intern. J. Algebra Comp. {\bf 19} (2009), 491--510.
\bibitem{b10} Ie.~Lutsenko, I.~Protasov, {\it Relatively thin and sparse subsets of groups}, Ukr. Math. J. {\bf 63} (2011), 216--225.
\bibitem{b11} V.~Malykhin, I.~Protasov, {\it Maximal resolvability of bounded groups}, Topology Appl. {\bf 20} (1996), 1--6.
\bibitem{b12} V.D.~Mazurov, E.I.~Khukhro (eds), {\it Unsolved problems in group theory, the Kourovka notebook}, 13-th augmented edition, Novosibirsk, 1995.
\bibitem{b13} I.V.~Protasov, {\it Ultrafilters and topologies on groups}, Siberian Math. J. {\bf 34} (1993), 163--180.
\bibitem{b14} I.V.~Protasov, {\em Selective survey on Subset Combinatorics of Groups}, Ukr. Math. Bull. {\bf 7} (2011), 220--257.
\bibitem{b15} I.V.~Protasov, {\it The combinatorial derivation}, Appl. Gen. Topology {\bf 14} (2013), N2, 171--178.
\bibitem{b16} I.V.~Protasov, T.Banakh, Ball Structures and Colorings of Groups and Graphs, Math. Stud. Monogr. Ser., Vol.~11, VNTL, Lviv, 2003.
\bibitem{b17} I.~Protasov, S.~Slobodianiuk, {\it Prethick subsets in partitions of groups}, Algebra Discrete Math. {\bf 14} (2012), 267--275.
\bibitem{b18} I.~Protasov, S.~Slobodianiuk, {\it Ultracompanions of subsets of groups}, Comment. Math. Univ. Corolin. {\bf 55} (2014), N2, 257--265.
\bibitem{b19} I.~Protasov, S.~Slobodianiuk, {\it Partitions of groups}, Matem. Stud. {\bf 42} (2014), 115--128.
\bibitem{b20} I.V.~Protasov, M.~Zarichnyi, General Asymptology, Math Stud. Monogr. Ser., Vol. 12, VNTL, Lviv 2007.
\end{thebibliography}
\end{document}